\documentclass[12pt]{article}

\usepackage{amsmath,amsthm,amsfonts,amssymb,amsthm}
\usepackage{hyperref, float}
\usepackage{tikz}
\usepackage{tikz-network}
\usepackage{pgf,tikz,pgfplots}
\pgfplotsset{compat=1.15}
\usepackage{mathrsfs}
\usetikzlibrary{arrows}

\usepackage{graphicx}
\graphicspath{{Images/}}

\newtheorem{theorem}{Theorem}[section]
\newtheorem{lemma}[theorem]{Lemma}
\theoremstyle{definition}
\newtheorem{definition}[theorem]{Definition}

\newtheorem{question}[theorem]{Question}
\newtheorem{conjecture}[theorem]{Conjecture}
\newtheorem{corollary}[theorem]{Corollary}
\newcommand{\ZZ}{\mathbb{Z}}

\newcommand{\PolyP}{\mathcal{P}}
\newcommand{\PolyPStar}{\mathcal{P}^*}
\newcommand{\PolySet}{\mathbb{P}}

\newcommand{\BoardB}{\mathbb{B}}
\newcommand{\FreePackingNumber}[1]{pn^{free}(#1)}
\newcommand{\FixedPackingNumber}[1]{pn^{fixed}(#1)}
\newcommand{\FixedEquiv}{\approx}
\newcommand{\FreeEquiv}{\sim}

\newcommand{\CPFI}[1]{cp^{fix}\left({#1}\right)} 
\newcommand{\CPFR}[1]{cp^{free}\left({#1}\right)} 
\newcommand{\CP}[1]{cp\left({#1}\right)} 

\begin{document}

\title{Clumsy Packing of Polyominoes in Finite Space}

\author{
Emma Miller, Moravian University \\
Mitchel O'Connor, Whitman College \\
Nathan Shank, Moravian University}

\date{\today}

\maketitle
\section{Introduction}


In 1971, Sands \cite{Sands71} proposed the following problem:  \textit{Given an $n \times n$ checkerboard, what is the minimum number of dominoes we can place on the board, assuming each domino covers exactly two adjacent squares, so that no additional domino can fit?} In other words, what is the maximum number of $1\times 1$ `holes' that can be created? In 1988, Gy\'{a}rf\'{a}s, Lehel, and Tuza \cite{Gyarfas88} called this type of packing a \textit{clumsy packing} which requires the density of packed polyominoes to be as small as possible relative to the size of the board. They were able to generate results for general board sizes with dominoes, which are $1\times 2$ rectangular polyominoes. Goddard \cite{Goddard95} considered clumsy packing density for \textit{hooks} (3 squares in the shape of an $L$), $m\times m$ \textit{squares}, and $1\times m$ \textit{longs} on infinite board. More recently, Walzer (\cite{Walzer14b}, \cite{WALZER14}) considered other clumsy packing of the infinite plane for more complicated polyominoes.  Tiling and packing of polyominoes has particular importance for topology and geometry with applications in physics and chemistry.  

In this paper, we consider a clumsy packing of polyominoes on a finite space. The polyominoes we consider are \textit{rectangular}, $L$, $T$, and $\textit{plus}$ polyominoes. If a polyomino is made of $n$ squares, we consider packing on a $n \times n$ board. This finite packing gives configurations and densities that are different than the previous results because of the edge effect that we have on a finite board.  Clumsy packing on finite boards has applications in product design related to filling a finite space using the least amount of material possible. 

\section{Background and Definitions}

        
    We will assume throughout that $a, b$, and $n$ are positive integers. We will be packing our polyominoes on a square grid of size $n \times n$ where $n$ is the size of the polyomino. We will adapt more of our notation and definitions from chapter 14 of \cite{book:Chap14}, however, to help with notation and terminology, we will make the following definitions:   
        
    \subsection{The Board}
    
    The board will be a square $n \times n$ grid where we place our polyominoes.  So, an $8 \times 8$ board can be visualized as a checkerboard.  A $1\times1$ closed square on our grid whose sides are parallel to the coordinate axes and corners are at integer coordinates will be called a \textit{cell} and we will denote the cell in column $i$ and row $j$ by $C_{i,j}$. For $1 \leq i, j \leq n$ define $X_i = \{C_{i, j}: 1 \leq j \leq n\}$ and $Y_j = \{C_{i, j}: 1 \leq i \leq n\}$.  We will refer to $X_i$ as column $i$ and $Y_j$ as row $j$. We will use the convention that $X_1$ is the leftmost column, $X_n$ is the rightmost column, $Y_1$ is the top row, and $Y_n$ is the bottom row. Thus, cell $C_{1,1}$ is the cell in the first column and first row which would be the upper left cell of the grid.  The square $n \times n$ grid will be called the board, $\BoardB$, so that $\BoardB = \{C_{i, j} : 1 \leq i \leq n \text{ and } 1 \leq j \leq n \text{ where } i,j \in \ZZ\}$.
        
    \subsection{General Polyominoes}
        A \textit{polyomino}, $\PolyP$, is a a finite set of cells. The number of cells in a polyomino will be the \textit{size} of the polyomino and will be denoted $|\PolyP|$. Note that two polyominoes are considered disjoint if they do not have any cells in common, but visually they may share edges or corners. Polyominoes will be located on board based on an \textit{anchor} which will be a specific cell that uniquely determines the location of the polyomino on the board. Informally, for example, for a $L$ polyomino (definition \ref{def:L polyomino}) we will define the anchor as the cell located on the intersection of the horizontal and vertical legs. 
        
        For integers $c$ and $d$ a \textit{shift} of a polyomino $\PolyP$ by $(c,d)$ is a polyomino $\PolyP_1$ so that $\PolyP_1 = \{C_{x+c, y+d}: C_{x,y} \in \PolyP\}$.  Thus $\PolyP_1$ is a shift of $\PolyP$ $c$ units horizontally and $d$ units vertically. We will use the notation $\PolyP_1 = \PolyP + (c,d)$ to represent a shift by $(c,d)$. 
        
        For certain polyominoes, we will consider a clockwise rotation of our polyomino by an integer multiple of $90^\circ$. For any integer $m$, if $\PolyP$ is a polyomino then $\PolyP R^m$ is a rotation of $\PolyP$ by $(m\cdot90)^\circ$ clockwise. This idea will be made more formal for each individual polyomino. Note that a rotation or shift may produce a polyomino which is no longer on the board. 
        
        We will not be considering reflections of polyominoes in this paper which could be of particular importance for $L$ polyominoes. We will be working with two different types of packings, free and fixed. For \textit{fixed packing}, we will not be allowed to rotate the polyomino, we will only be allowed to shift the polyomino.  In \textit{free packing}, we will be allowed to rotate and shift the polyomino.  
        
        \begin{definition}
            We will call two polyominoes $\PolyP_1$ and $\PolyP_2$ \textit{fixed equivalent} if there is an integer pair $(c,d)$ so that $\PolyP_2= \PolyP_1+(c,d)$. We will denote fixed equivalent polyominoes by $\PolyP_1 \FixedEquiv \PolyP_2$.
        \end{definition} 
        
        Thus, two polyominoes are fixed equivalent if we can shift one polyomino by $(c,d)$ to equal the other polyomino.
        
        For free equivalent polyominoes, we are allowed to rotate and shift the polyominoes. 
        
        \begin{definition}
            We will call two polyominoes $\PolyP_1$ and $\PolyP_2$ \textit{free equivalent} if there is a non-negative integer $m$ and and integers pair $(c,d)$ so that $\PolyP_2 = \PolyP_1 R^m + (c,d)$. We will denote free equivalent polyominoes by $\PolyP_1 \FreeEquiv \PolyP_2$.
        \end{definition}
        
        Our objective is to pack as few copies of our polyomino on the board so that we can not include another polyomino.  To do so, we need to define how polyominoes fit onto a board and what makes a valid arrangement.  
        
    \begin{definition}
        Given a board $\BoardB$ and a set of polyominoes $\PolySet = \{\PolyP_1, \ldots \PolyP_k\}$, we say $\PolySet$ is a \textit{valid arrangement} if $\bigcup_{i=1}^k \PolyP_i \subseteq \BoardB$ and if $\PolyP_i \cap \PolyP_j = \emptyset$ for all $1\leq i < j \leq k$. Otherwise, we call $\PolySet$ an \textit{invalid arrangement.}  
    \end{definition}
    
    Therefore a valid arrangment simply means the the set of polyominoes is pairwise disjoint and fits onto the board.  

    The following definitions are for free polyominoes and can easily be extended to fixed polyominoes.   
    
        \begin{definition}
            A set of polyominoes composed of polyominoes which are free equivalent polyominoes to $\PolyP$ will be called a \textit{free polyomino set of $\PolyP$.}
        \end{definition}

        \begin{definition}
            A free polyomino set of $\PolyP$, denoted $\PolySet$, is a \textit{free packing of $\PolyP$} on a board $\BoardB$ if $\PolySet$ is a valid arrangement such that for any free equivalent polyomino $\PolyPStar \FreeEquiv \PolyP$ we have that $\PolyPStar \cup \PolySet$ is an invalid arrangement. Thus, we can think of a free packing as a maximal set of $\PolyP$. 
        \end{definition}
        
    So for a set of polyominoes to be a free packing, they must be a valid packing and they must be maximal in the sense that no additional polyomino can be placed onto the board without overlapping an already placed polyomino.  The number of polyominoes in the set is the free packing number of the set.  
        
        \begin{definition}
            Given a free packing of $\PolyP$, denoted $\PolySet$, the size of $\PolySet$ is the \textit{free packing number} of $\PolySet$. We will denote this by $\FreePackingNumber{\PolySet}$.
        \end{definition}
    
        \begin{definition}
            The \textit{clumsy free packing number of $\PolyP$}, denoted $\CPFR{\PolyP}$ is the minimum free packing number over all free polyomino sets of $\PolyP$.  
        \end{definition}
        
    The clumsy free packing number is the minimum of all of the possible free packing numbers for a particular polyomino.  In essence, the clumsy free packing number is the minimum of the maximals.  Any set of polyominoes which attains this minimum will be called a clumsy free packing.
        
        \begin{definition}
            A \textit{clumsy free packing of $\PolyP$} is a free polyomino set of $\PolyP$, denoted $\PolySet$, whose free packing number is the clumsy free packing number of $\PolyP$.  Thus $\FreePackingNumber{\PolySet} = \CPFR{\PolyP}$. 
        \end{definition}  
 
    Unless otherwise specified, we will always consider packing (free or fixed) of a polyomino $\PolyP$ on a $n \times n$ board $\BoardB$ where $n=|\PolyP|$.         

\section{Results}
    This section will consist of subsection for \textit{rectangular} polyominoes, $L$ polyominoes, $T$ polyominoes, and \textit{plus} polyominoes. In each section, we will define the polyomino and consider fixed and free packing.  
    
    \subsection{Rectangular Polyominoes}
         \textit{Rectangular polyominoes} are a popular polyomino to study.  The classic domino is a rectangular polyomino.  Some packing problems have been studied for dominoes and straights, which are $1\times m$ \textit{rectangular} polyominoes.  
         
            \begin{definition}\label{def:rectangular polyomino}
                The \textit{rectangular polyomino}, $R_{a,b}$, is the polyomino consisting of all cells $C_{i,j}$ where $1 \leq i \leq a$ and $1 \leq j \leq  b$.  Thus $R_{a,b}$, is a rectangular polyomino of size $ab$.  
            \end{definition}
        
        To know the location of a rectangular polyomino, we need to define the anchor. 
        
            \begin{definition}\label{def:anchor}
                Given a rectangular polyomino, $R_{a,b}$, with $a, b >1$, we will define the \textit{anchor} of the polyomino to be the cell $C_{\lfloor \frac{a}{2}\rfloor, \lfloor \frac{b}{2}\rfloor}$. 
            \end{definition}
            
            Notice that if $a=b=1$ then $R_{1,1}$ is just an individual cell which is a trivial case since $\CPFI{R_{1,1}} = \CPFR{R_{1,1}} = 1.$
        
        \textit{Straight polyominoes} are a special case of rectangular polyominoes. We will use straight vertical and straight horizontal polyominoes for our clumsy free packing.  
        
            \begin{definition}\label{def:straight vertical polyomino}
                For $n>1$, the \textit{straight vertical polyomino}, $SV_{n}$, is the rectangular polyomino $R_{1,n}$. We will define the anchor as $C_{1, \lfloor \frac{n}{2} \rfloor}$. 
            \end{definition}
        
            \begin{definition}\label{def:straight horizontal polyomino}
                For $n>1$, the \textit{straight horizontal polyomino}, $SH_{n}$, is the rectangular polyomino $R_{n,1}$. We will define the anchor as $C_{\lfloor \frac{n}{2}\rfloor, 1}$
            \end{definition}
        
        The following theorem finds the clumsy packing number for fixed and free straight polyominoes.  
        
            \begin{theorem}\label{thm:straight polyominoes} 
                For any $n\geq 1$,  $$\CPFR{SV_n}=\CPFR{SH_n} = \CPFI{SV_{n}}=\CPFI{SH_{n}} = n.$$ 
            \end{theorem}
        
            \begin{proof} 
                \textbf{Fixed Case:}  We will only consider straight vertical polyominoes, $SV_{n}$.  Notice that any polyomino $\PolyP_1 \approx SV_n$ (fixed equivalent) will consist of cells from only one column $X_i$.  So a valid arrangement of $1<m<n$ polyominoes fixed equivalent to $SV_n$ will have $n-m$ empty columns, which is not a fixed packing. Therefore the clumsy packing number must be at least $n$, but it clearly can not be greater than $n$ since there are only $n^2$ cells in the board. Thus $\CPFI{SV_n} = n$.  
            
            
                \noindent\textbf{Free Case:} Notice that any rotation of a straight polyomino will result in a straight polyomino.  Also, any valid arrangement on a $n \times n$ board cannot contain a straight vertical polyomino and a straight horizontal polyomino as they cannot be disjoint. Therefore $\CPFR{SV_n} = \CPFI{SV_n}$. 
                An identical argument can be made for $SH_{n}$ using rows.

            \end{proof}
        
        The following theorem finds the clumsy fixed packing number for \textit{rectangular} polyominoes of size at least $2 \times 2$.  
        
            \begin{theorem}\label{thm:fixed rectangular where a,b geq 2} 
                For any rectangular polyomino, $R_{a,b}$ with $a, b \geq 2$, the clumsy fixed packing is $\CPFI{R_{a,b}} = \left\lceil \frac{ab-a+1}{2a-1} \right\rceil \left\lceil \frac{ab-b+1}{2b-1} \right\rceil$.
            \end{theorem}
        
            \begin{proof}
                Notice that $n = ab$. In order to find the clumsy fixed packing number for rectangular polyominoes, we will first show that there exists a fixed packing of a set of $R_{a,b}$, $\PolySet$, so that $\FixedPackingNumber{\PolySet} =  \left\lceil \frac{ab-a+1}{2a-1} \right\rceil \left\lceil \frac{ab-b+1}{2b-1} \right\rceil.$  
                
                Consider tiling the board by having $b-1$ empty rows, $Y_1$ through $Y_{b-1}$, and $a-1$ empty columns, $X_1$ through $X_{a-1}$. Then have a horizontal tiling of polyominoes with $a-1$ empty column between each polyomino and $b-1$ empty rows between each horizontal tiling. We may need to add an additional column of polyominoes so that the number of empty columns to the right of the board is less than or equal to $a-1$.  Similarly, we may need to add an additional row of polyominoes so that the number of empty rows at the bottom of the board is less than or equal to $b-1$. See Figure \ref{fig:RectPolyUpperBound}.
                
                \begin{figure}[h!]
                    \centering
                    \includegraphics[width = .9\linewidth]{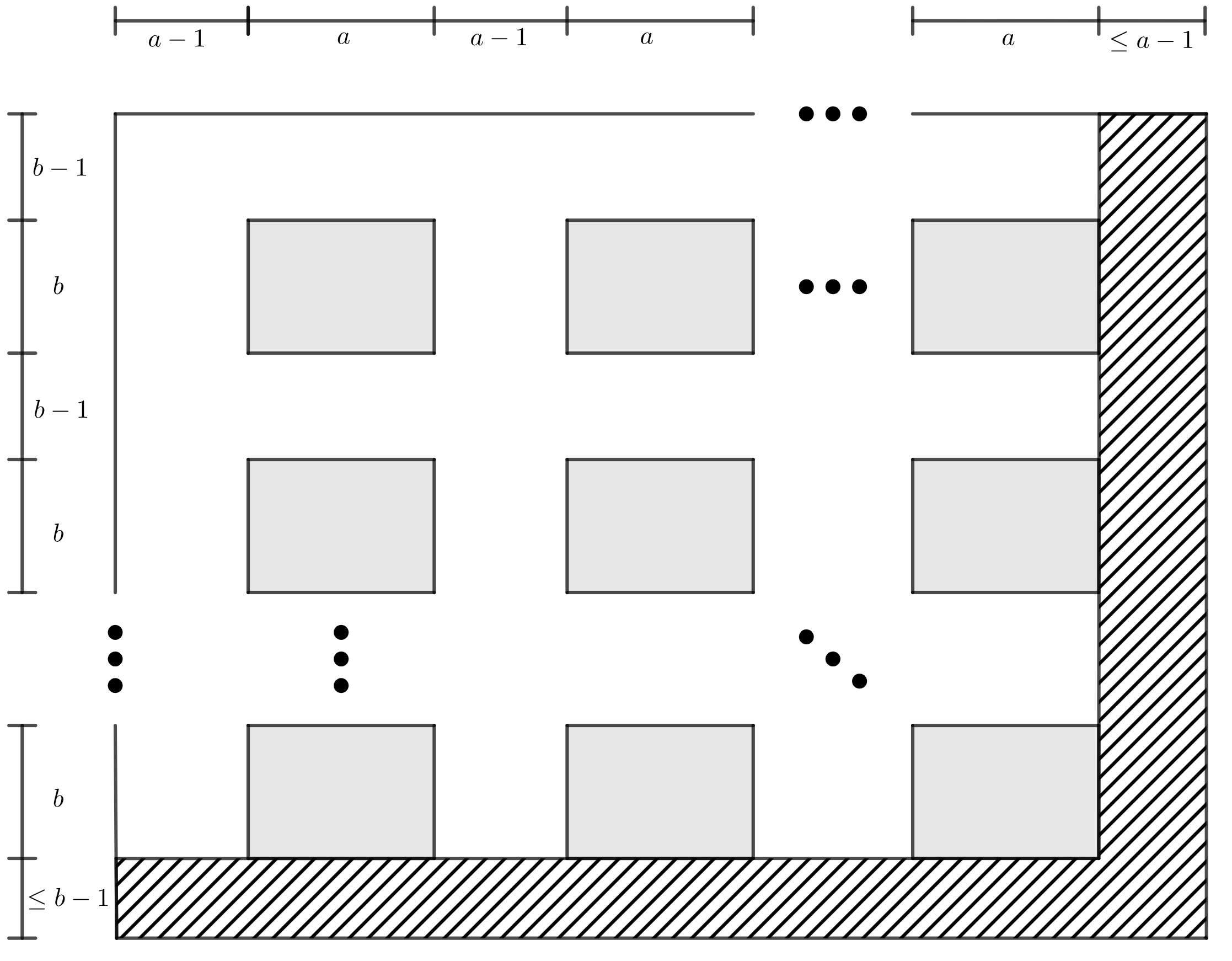}
                    \caption{This figure depicts a general example of a clumsy fixed packing of some $R_{a,b}$. The shaded areas are the polyominoes. The hatched region represents the area that may or may not include an additional row or column of polyominoes.}
                    \label{fig:RectPolyUpperBound}
                \end{figure}
            
                The number of polyominoes placed along a non-empty row will be $\left\lceil\frac{ab-(a-1)}{2a-1}\right\rceil$.  The number of polyominoes placed along a non-empty column will be $\left\lceil\frac{ab - (b-1)}{2b-1}\right\rceil$. So the total number of polyominoes in this particular tiling will be $\left\lceil \frac{ab-a+1}{2a-1} \right\rceil \left\lceil \frac{ab-b+1}{2b-1} \right\rceil$. Therefore $$\CPFI{R_{a,b}} \leq \FixedPackingNumber{\PolySet} = \left\lceil \frac{ab-a+1}{2a-1} \right\rceil \left\lceil \frac{ab-b+1}{2b-1} \right\rceil.$$
            
                Now we will show that this is also a lower bound for $\CPFI{R_{a,b}}$ and we will do this based on the parity of $a$ and $b$. 

                \textbf{Even/Odd Case}:  Assume we have a clumsy fixed packing of $R_{a,b},$ where $a,b \geq 2$ and $a=2k$ and $b=2l+1$ for non-negative integers $k$ and $l$. 

                \begin{figure}[h!]
                    \centering
                    \includegraphics[width=.7\linewidth]{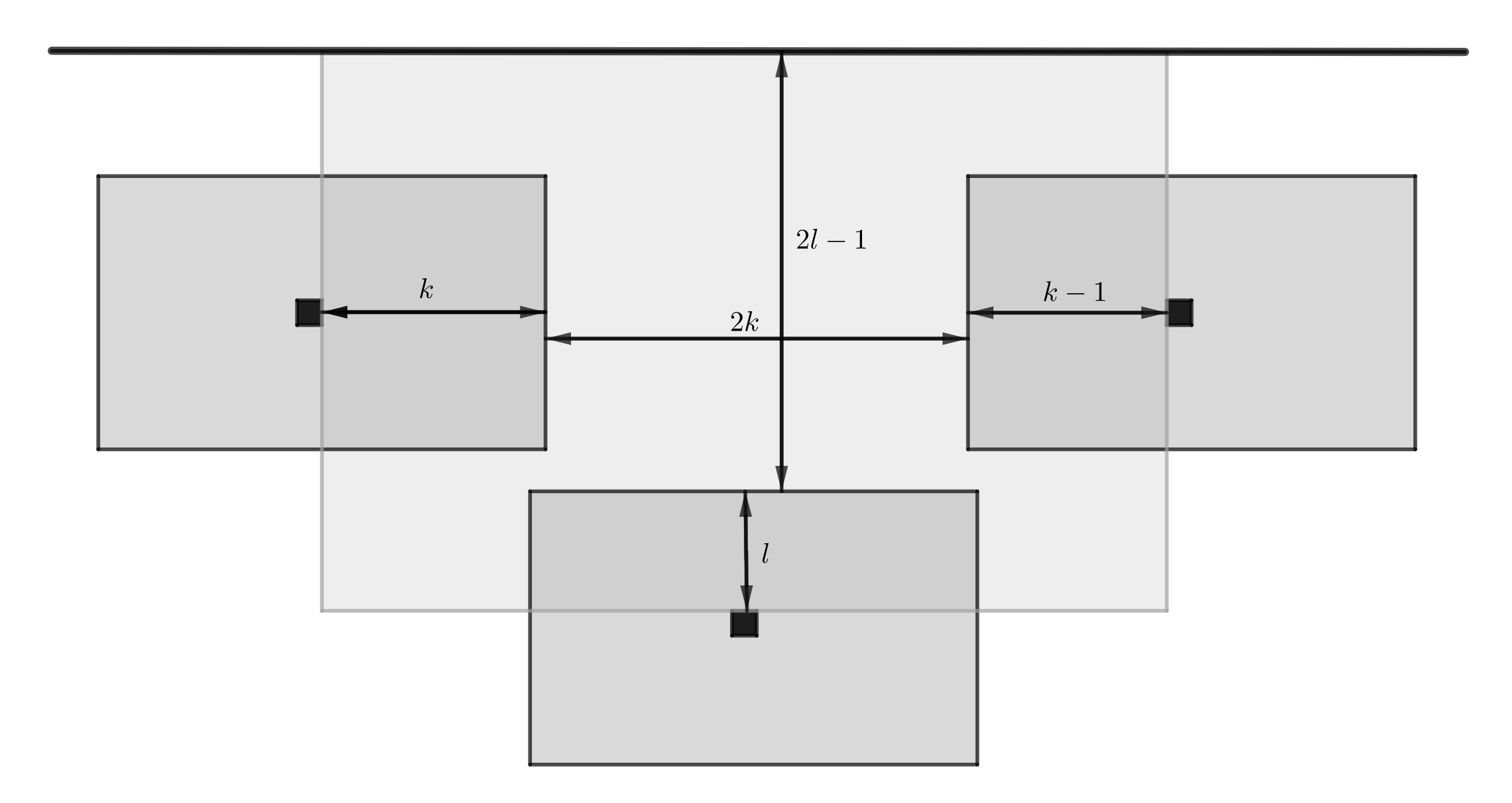}
                    \caption{Even by Odd Case: This shows that a $(4k-1) \times (3l+1)$ rectangular region along the top of the board must contain an anchor which are represented by the black cells. If not, there is a $(2k) \times (2l+1)$ region which can contain another polyomino.  
                    }
                    \label{fig:Even_Odd_Boundry}
                \end{figure}
            
                It is easy to see that there must be an anchor in the top $3l+1$ rows and the bottom most $3l+1$ rows and every $4l+1$ rows in between. Similarly, there also must be an anchor in the leftmost $3k-1$ column the rightmost $3k$ columns and every $4k-1$ columns in between. Note that every region of $4l+1$ rows and $4k-1$ columns must contain an anchor. Otherwise, in any region of such size without an anchor, we can fit an additional polyomino fixed equivalent to $R_{a,b}$. This is demonstrated in Figure \ref{fig:Even_Odd_Boundry} across the top edge of the board. 
                Therefore, there are $2 + \left\lfloor \frac{ab - 2(3l+1)}{4l+1} \right\rfloor=2 + \left\lfloor \frac{ab - (3b-1)}{2b-1} \right\rfloor$ pairwise disjoint sets of rows, each of which must contain an anchor in each of the $2+\left\lfloor \frac{ab - (3k-1+3k)}{4k-1} \right\rfloor = 2+\left\lfloor \frac{ab - (3a-1)}{2a-1} \right\rfloor $ pairwise disjoint sets of columns. This leaves us with a total of $ \left\lfloor 2+ \frac{ab - (3b-1)}{2b-1} \right\rfloor \left\lfloor 2+ \frac{ab - (3a-1)}{2a-1} \right\rfloor$ disjoint regions on the board each of which must contain an anchor. 
                
                (Note that if $\frac{ab-2(3l+1)}{4l+1}<0$, then the top $3l+1$ rows and bottom  $3l+1$ rows are not disjoint and therefore we can place one anchor which will be in the top most $3l+1$ rows and the bottom $3l+1$ rows.  Also, notice that since $a\geq 2$ we know $ab\geq 2l+1$ which implies $\frac{ab-2(3l+1)}{4l+1} \geq -1$. Similarly for columns.)
            
                This shows that for the even/odd case we have 
                \begin{equation*}\label{equation:CPFIXRECT}
                    \CPFI{R_{a,b}} \geq \left\lfloor 2+ \frac{ab - (3b-1)}{2b-1} \right\rfloor \left\lfloor 2+ \frac{ab - (3a-1)}{2a-1} \right\rfloor.
                \end{equation*}

                \textbf{Even/Even Case}: Assume $a,b \geq 2$ and $a=2k$ and $b=2l$. Following similar logic as above, we can see there must be an anchor in the top $3l-1$ rows, bottom $3l$ rows, and every $4l-1$ rows in between.  There also must be an anchor in the leftmost $3k-1$ columns, the rightmost $3k$ columns, and every $4k-1$ columns in between.  Note that every region of $4k-1$ columns and $4l-1$ rows must contain an anchor.  
                
                Therefore, there are $2 + \left\lfloor \frac{ab - (3l-1)-3l}{4l-1} \right\rfloor=2 + \left\lfloor \frac{ab - (3b-1)}{2b-1} \right\rfloor$ pairwise disjoint sets of rows, each of which must contain an anchor in each of the $2+\left\lfloor \frac{ab - 3k-(3k-1)}{4k-1} \right\rfloor = 2+\left\lfloor \frac{ab - (3a-1)}{2a-1} \right\rfloor $ pairwise disjoint sets of columns. This leaves us with a total of $ \left\lfloor 2+ \frac{ab - (3b-1)}{2b-1} \right\rfloor \left\lfloor 2+ \frac{ab - (3a-1)}{2a-1} \right\rfloor$ disjoint regions on the board each of which must contain an anchor. Notice this is the same as in the even/odd case. 
                
                \textbf{Odd/Odd Case}:  Assume $a,b \geq 2$ and $a=2k+1$ and $b=2l+1$. Following the similar logic as above, we can see there must be an anchor in the top $3l+1$ rows, bottom $3l+1$ rows, and every $4l+1$ rows in between.  There also must be an anchor in the leftmost $3k+1$ columns, the rightmost $3k+1$ columns, and every $4k+1$ columns in between.  Note that every region of $4k+1$ columns and $4l+1$ rows must contain an anchor.  
                
                Therefore, there are $2 + \left\lfloor \frac{ab - (3l+1)-(3l+1)}{4l+1} \right\rfloor=2 + \left\lfloor \frac{ab - (3b-1)}{2b-1} \right\rfloor$ pairwise disjoint sets of rows, each of which must contain an anchor in each of the $2+\left\lfloor \frac{ab - (3k+1)-(3k+1)}{4k+1} \right\rfloor = 2+\left\lfloor \frac{ab - (3a-1)}{2a-1} \right\rfloor $ pairwise disjoint sets of columns. This leaves us with a total of $ \left\lfloor 2+ \frac{ab - (3b-1)}{2b-1} \right\rfloor \left\lfloor 2+ \frac{ab - (3a-1)}{2a-1} \right\rfloor$ disjoint regions on the board each of which must contain an anchor. Notice this is the same as in the previous two cases. 
                
                Since all three cases give the same inequality, we will now proceed independent of the parity of $a$ and $b$.

        
                Now we need to show the equality of our upper and lower bound on the fixed clumsy packing number. 
                
                To accomplish this, we call upon a well known equality for floor and ceiling functions: 
        
                \begin{equation*}
                    \left\lceil \frac{T}{B}\right\rceil =\left\lfloor \frac{T+B-1}{B}\right\rfloor.
                \end{equation*}
                
                From this we can easily show that,
                \begin{equation*}
                    \left\lceil \frac{T}{B}\right\rceil =\left\lfloor \frac{T-1-B}{B}\right\rfloor +2 .
                \end{equation*}
                
                With this we are able to see
        
                \begin{align*}
                    \left\lceil \frac{ab-a+1}{2a-1} \right\rceil
                    &= \left\lfloor\frac{(ab-a+1)-1-(2a-1)}{2a-1}\right\rfloor+2\\
                    &= \left\lfloor\frac{ab-(3a-1)}{2a-1}\right\rfloor+2
                \end{align*}
                
                and similarly 
                \begin{align*}
                    \left\lceil \frac{ab-b+1}{2b-1} \right\rceil
                    &= \left\lfloor\frac{ab-(3b-1)}{2b-1}\right\rfloor+2.
                \end{align*}
                
                Thus we are able to justify that
                $$\CPFI{R_{a,b}} \geq \left\lceil \frac{ab-a+1}{2a-1} \right\rceil \left\lceil \frac{ab-b+1}{2b-1} \right\rceil.$$
            \end{proof}
        
        The clumsy free packing for general \textit{rectangular} polyominoes turned out to be more difficult because of the edge effects. We pose this as an open question in Section \ref{sect:questions}. 
        
    \subsection{$L$ Polyominoes}
    
    As noted in the introduction, previous work has been completed for \textit{hooks} which is a special case of a $L$ polyomino.  
    
    
            \begin{definition}\label{def:L polyomino}
                An \textit{L polyomino}, $L_{a,b}$, is a polyomino where $0 < a \leq b$ such that $L_{a,b} = \{C_{i,1} : 1 \leq i \leq a+1\} \cup\{C_{1,j} : 1 < j \leq b+1 \}$. The $|L_{a,b}| = a + b + 1$.  We will define the \textit{anchor} of the polyomino to be the cell $C_{1,1}$.
            \end{definition}
        
        Thus, $L$ polyominoes have legs of length $a+1$ and $b+1$ and have size $a+b+1$. 
            
        We will define $LR_{a,b}$, $LR^2_{a,b}$, and $LR^3_{a,b}$ to be a $90^{\circ}, 180^{\circ}$, and $270^{\circ}$ clockwise rotation of $L_{a,b}$ respectively.  So we have:
        \begin{itemize}
            \item $LR_{a,b} = \{C_{i,1} : 1 \leq i \leq b+1\} \cup \{C_{b+1,j} : 1 \leq j \leq a+1 \}$
        which will have anchor at cell $C_{b+1, 1}$,
            \item $LR^2_{a,b} = \{ C_{i,b+1} : 1 \leq i \leq a+1\} \cup \{C_{a+1,j} : 1 \leq j \leq b+1 \}$
        which will have anchor at cell $C_{a+1, b+1}$, and 
            \item $LR^3_{a,b} = \{ C_{i,a+1} : 1 \leq i \leq b+1\} \cup \{C_{1,j} : 1 \leq j \leq a+1 \}$
        which will have anchor at cell $C_{1, a+1}.$ 
        \end{itemize}
        An example of an $L$ polyomino rotated is provided in Figure \ref{fig:rotations}.
    
        \begin{figure}[h!]
            \centering \includegraphics[width=.8\linewidth]{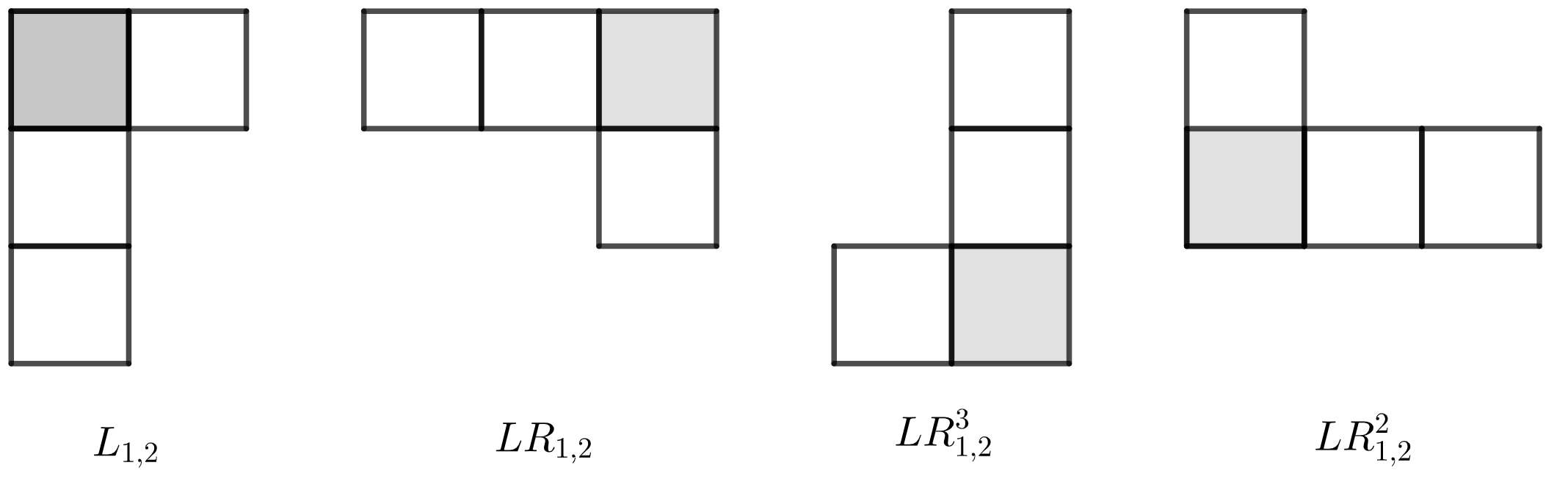}
            \caption{This figure demonstrates $L_{1,2}$, $LR_{1,2}$, $LR^2_{1,2}$, and $LR^3_{1,2}$. The shaded cell is the anchor.}
            \label{fig:rotations}
        \end{figure}
        
        For simplicity, we will use $L_{a,b}, LR_{a,b},$ etc. to describe the size and orientation of the polyomino and use the anchor to describe its location which is just a shift of $L_{a,b}, LR_{a,b}$, etc..   
        
        
        The packing of $L$ polyominoes may seem straight forward, but there are some intricacies we must consider.  The case where $a=b$ is a special case as seen in Theorem \ref{thm:fixed L polyominoes a=b} for clumsy fixed packing and Theorem \ref{thm:free L polyominoes a=b} for clumsy free packing. 
        
            \begin{theorem}\label{thm:fixed L polyominoes a=b} 
                For a $L$ polyomino, $L_{a,a}$, $\CPFI{L_{a,a}} = 1$.
            \end{theorem}
            
            \begin{proof}
                Place the $L$ polyomino so that the anchor is on the cell $C_{a+1,1}$ as seen in Figure \ref{fig:fixedLequal}.  Notice that if we want to include another $L$ polyomino, we must have the anchor in the first $a$ rows and $a$ columns.  However, we must have at least $a+1$ consecutive empty horizontal cells which does not happen in the first $a$ rows.  
    
                \begin{figure}[h!]
                    \centering
                    \includegraphics[width=.4\linewidth]{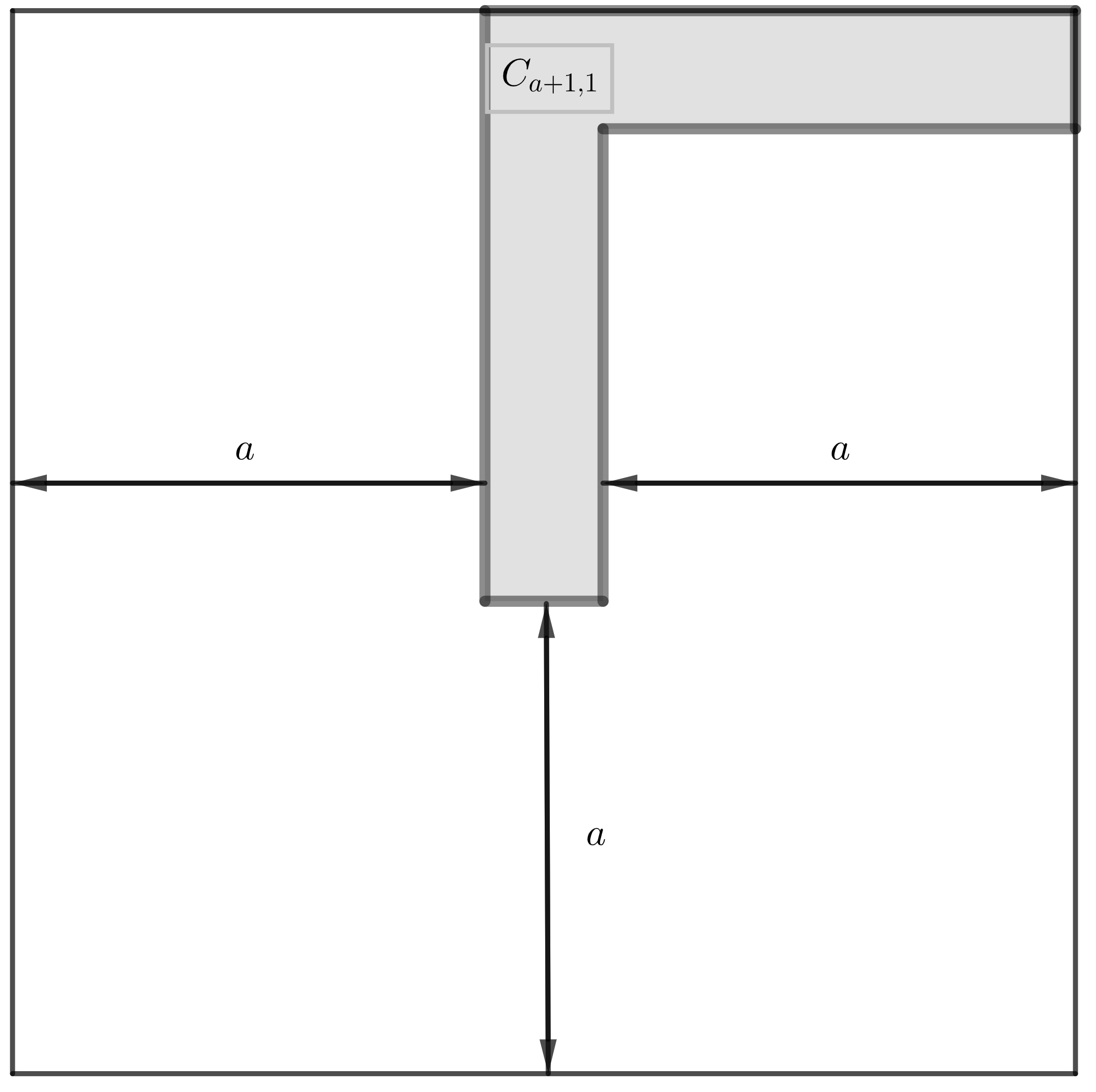}
                    \caption{This figures shows the size of the empty regions after placing the anchor of $L_{a,a}$ in $C_{a+1,1}$.}
                    \label{fig:fixedLequal}
                \end{figure}
            \end{proof}
            
        A general formula for $\CPFI{L_{a,b}}$ is conjectured in section \ref{sect:questions}.  Notice that $\CPFI{L_{a,b}}$ could be small, as in the case when $a=b$.  However when $a$ and $b$ are dramatically different, for example $a=1$ and $b>>a$, then every set of three adjacent columns must contain an anchor. So our clumsy fixed packing number is at least $\lfloor \frac{b}{3} \rfloor$.  
        
        Now we turn our attention to free packing.  We provide bounds for $L_{a,b}$ as well as show that $\CPFR{L_{a,a}} = 2.$
            
            \begin{lemma}\label{lem:freeL2}
                For a $L$ polyomino, $L_{a,b}$,  $\CPFR{L_{a,b}} \geq 2.$
            \end{lemma}
            
            \begin{proof}
                It suffices to show that there does not exist a clumsy packing of size 1.  
                Assume by contradiction, there exist two positive integers $a$ and $b$ so that $\CPFR{L_{a,b}} = 1$. Without loss of generality, we will use $L_{a,b}$ without any rotations.  
                
                Assume the anchor is on cell $C_{x,y}$. If $y>1$ and $x>1$ then we can fit another $L_{a,b}$ with anchor at cell $C_{x-1, y-1}$.  
                Assume $y = 1$. We can fit an $LR^2_{a,b}$ with anchor at cell $C_{x+a, b+2}$. A similar argument holds for $x=1.$ See Figure \ref{fig:LFreeGEQ2} for an example. 
                
                 \begin{figure}[h!]
                    \centering
                    \includegraphics[width=.4\linewidth]{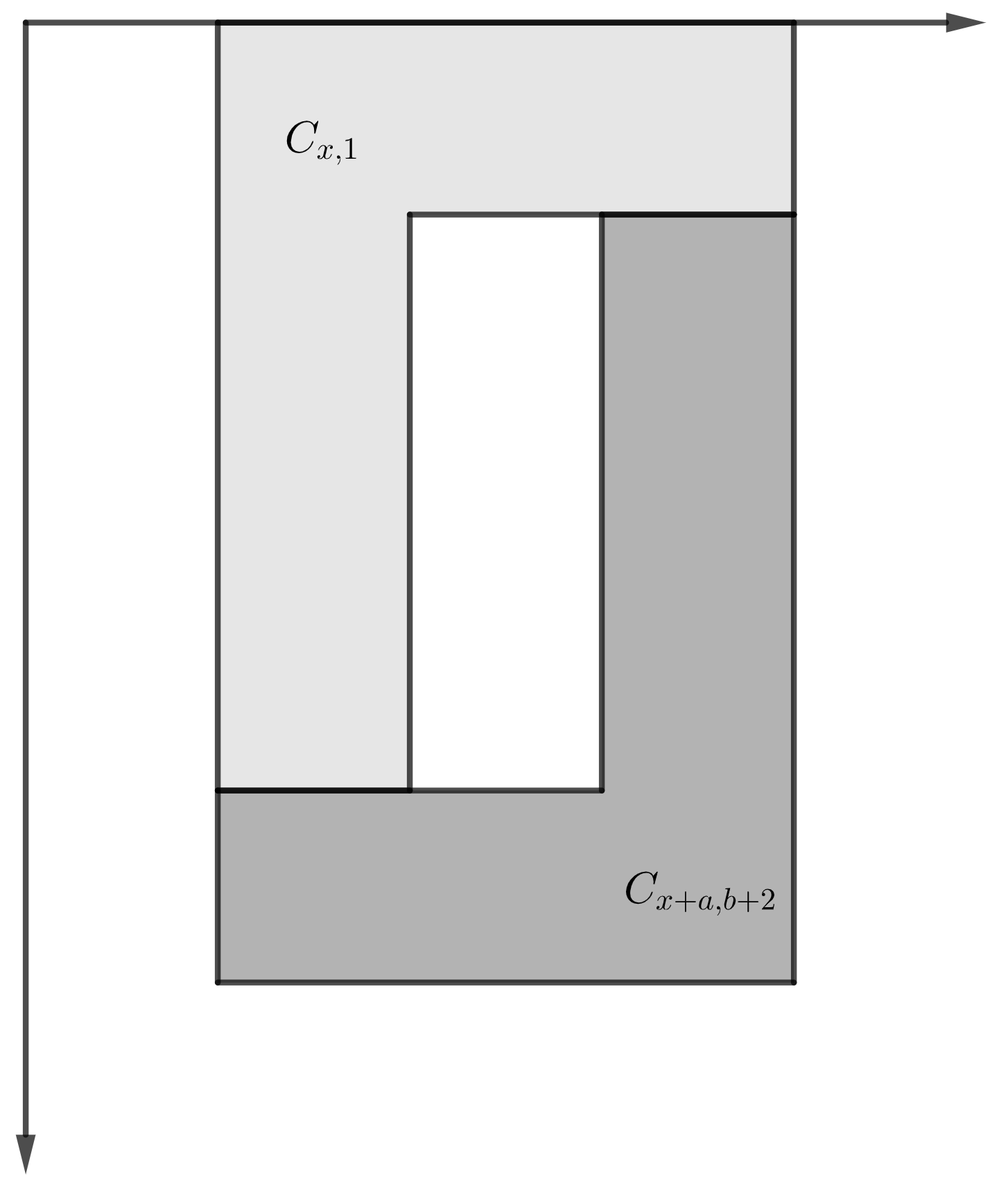}
                    \caption{If $y=1$, then we can fit $LR^2_{a,b}$ with anchor at cell $C_{x+a, b+2}$.}
                    \label{fig:LFreeGEQ2}
                \end{figure}
            \end{proof}
        
            \begin{theorem}\label{thm:free L polyominoes a=b}
                For a $L$ polyomino, $L_{a,a}$, $\CPFR{L_{a,a}} = 2.$
            \end{theorem}
            
            \begin{proof}
                By Lemma \ref{lem:freeL2}, $\CPFR{L_{a,a}} \geq 2$. Consider placing a $L$ polyomino $L_{a,a}$, call it $L^1$, with anchor at cell $C_{a+1,a+1}$. Place another $L$ polyomino $L_{a,a}$, call it $L^2$, with anchor at cell $C_{a, 1}$. Notice to the left of $L^1$ and $L^2$ there are at most $a$ empty columns which is not enough for another $L$ polyomino. Similarly, notice to the right of the $L^1$ and $L^2$, there are at most $a$ empty rows which is not enough for another $L$ polyomino. As seen in Figure \ref{fig:fixedLa=b}.
                
                \begin{figure}[h!]
                    \centering
                    \includegraphics{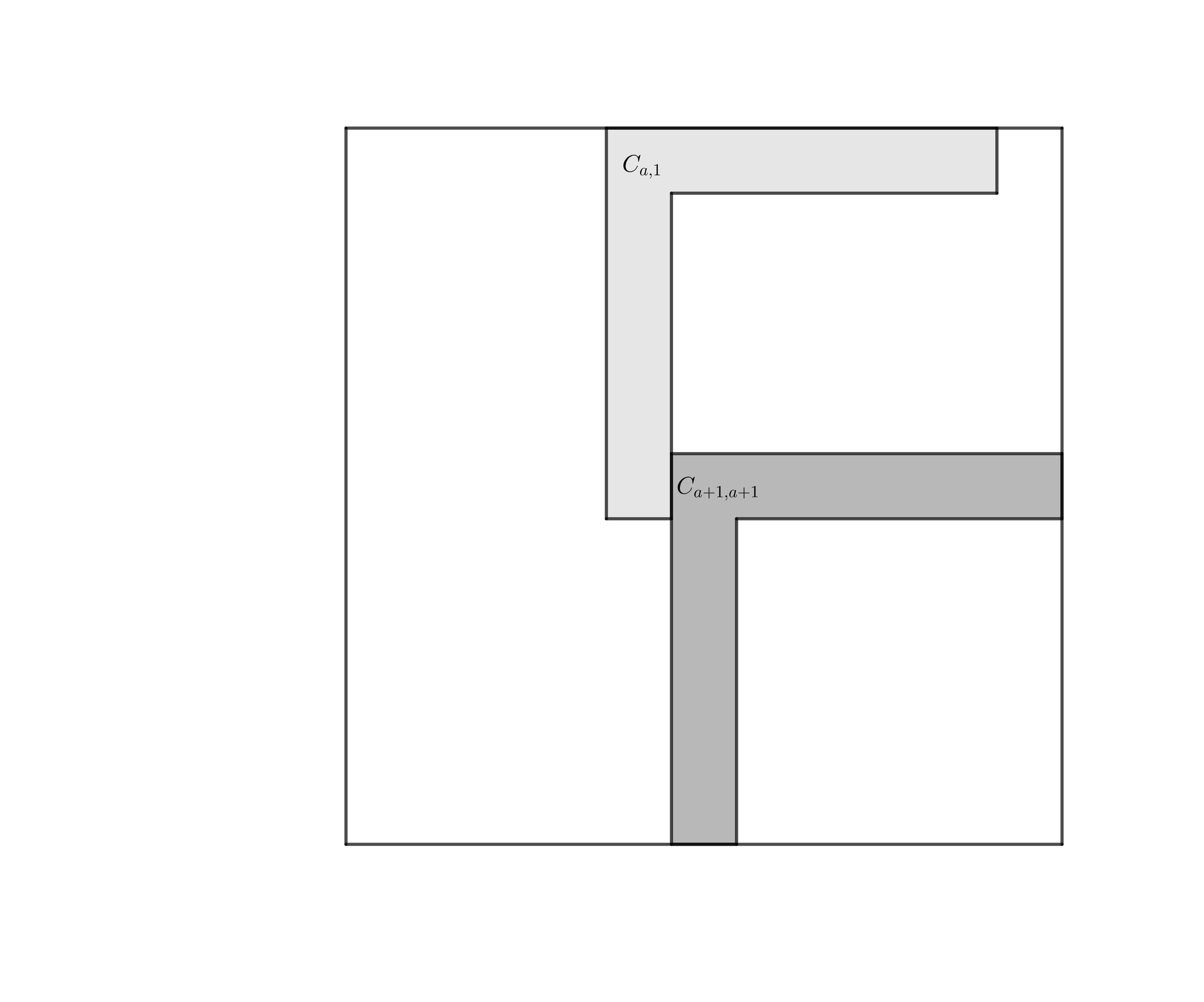}
                    \caption{Free packing showing $\CPFR{L_{a, a}} = 2$.}
                    \label{fig:fixedLa=b}
                \end{figure}
            \end{proof}
        
        For a general $L$ polyomino, the maximum number of free $L$ polyominoes we need is 5.  This is shown in the following theorem. 
        
            \begin{theorem}\label{thm:free L polyominoes}
                For a $L$ polyomino, $L_{a,b}$,  $2 \leq \CPFR{L_{a,b}} \leq 5$.
            \end{theorem}
            
            \begin{proof}
                By Lemma \ref{lem:freeL2}, $2 \leq \CPFR{L_{a,b}}$.
                A valid arrangement of five $L$ polyominoes can always be constructed in the manner shown in Figure \ref{fig:5LPacking}, where the four $L$ polyominoes on the left are embedded into a square of sizelength $b+2$.  Thus, there can not be another $L$ polyomino in the center of these four $L$ polyominoes. The bottom most  $a-1$ rows and right most $a-1$ columns will be empty. However, placing a $LR^2_{a,b}$ with anchor at $C_{a+b+1, b+3}$ will create a free packing since $a \leq b$.  Therefore, $2 \leq \CPFR{L_{a,b}} \leq 5$.
            
                \begin{figure}[h!]
                    \centering
                    \includegraphics[width = .4\linewidth]{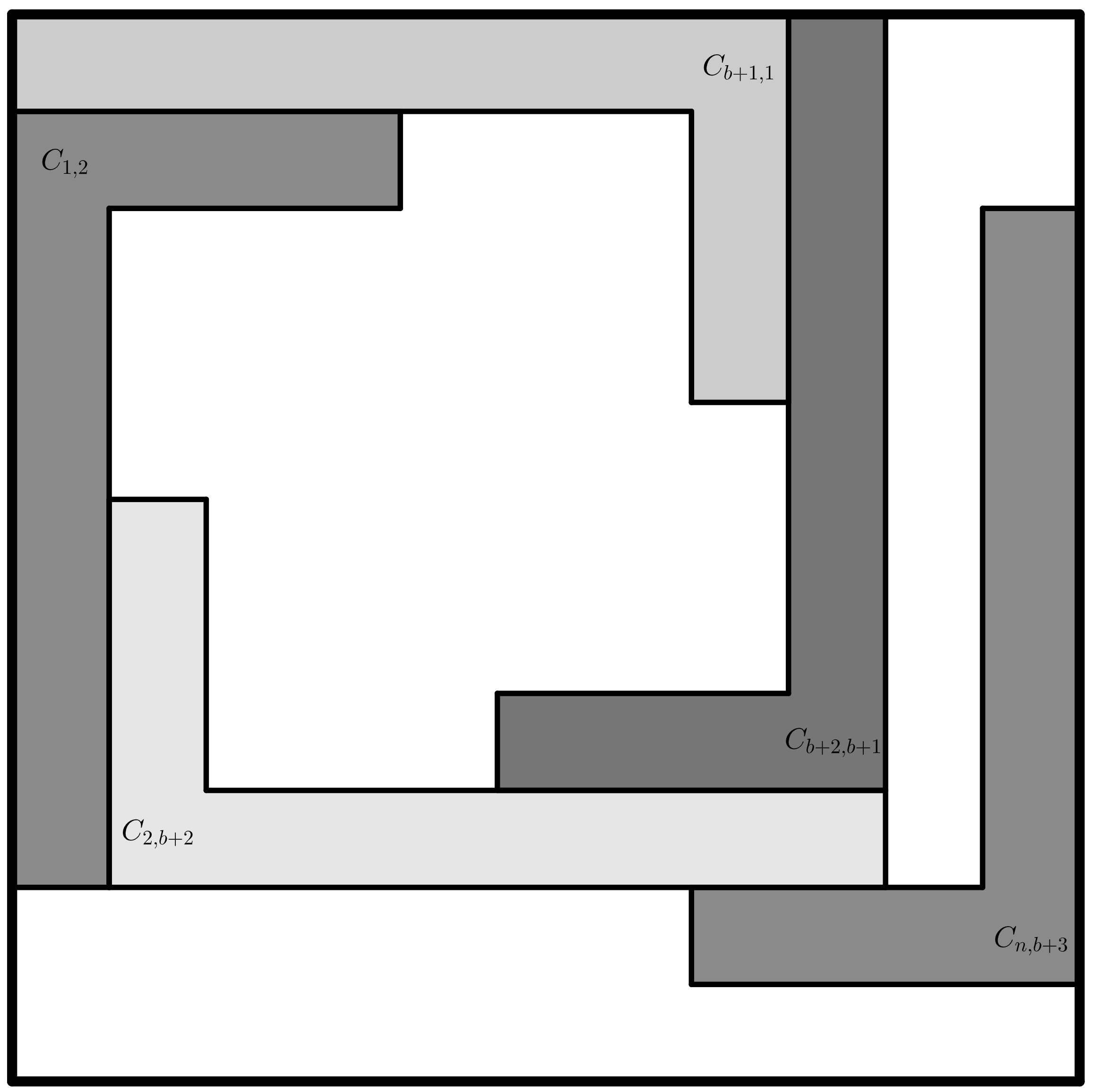}
                    \caption{General pattern to show $\CPFR{L_{a,b}} \leq 5$.}
                    \label{fig:5LPacking}
                \end{figure}
            \end{proof}
        
        With these bounds, we know the clumsy free packing number is often less than the upper bound.  For example, consider $a=3, b=6$ and $n=10.$ We can place three copies $L_{3, 6}$ with anchor at $C_{2, 1}, C_{3, 4},$ and $C_{7,4}$.  This creates a free packing of $L_{3,6}$ whose packing number is 3.  A similar arrangement can be made to find a free packing of $L_{2, 7}$ whose packing number is 4. We conjecture that these are the clumsy free packing numbers for $L_{3,6}$ and $L_{2,7}$ respectively.  
        
            \begin{corollary}\label{col:free L polyominoes a=1}
                For a $L$ polyomino, $L_{1,b}$, $2 \leq \CPFR{L_{1,b}} \leq 4$.
            \end{corollary}
        
            \begin{proof}
                Note that if $a = b = 1$ then $\CPFR{L_{1,1}} = 2$ via Theorem \ref{thm:free L polyominoes a=b}. Given $1 \ne b$, a free packing of four $L$ polyominoes can always be constructed in the manner shown in Figure \ref{fig:4LPacking}, which has at most $b$ consecutive empty horizontal cells and $b$ consecutive empty vertical cells. Therefore, $2 \leq \CPFR{L_{1,b}} \leq 4$.
                
                \begin{figure}[h!]
                    \centering
                    \includegraphics[width = .4\linewidth]{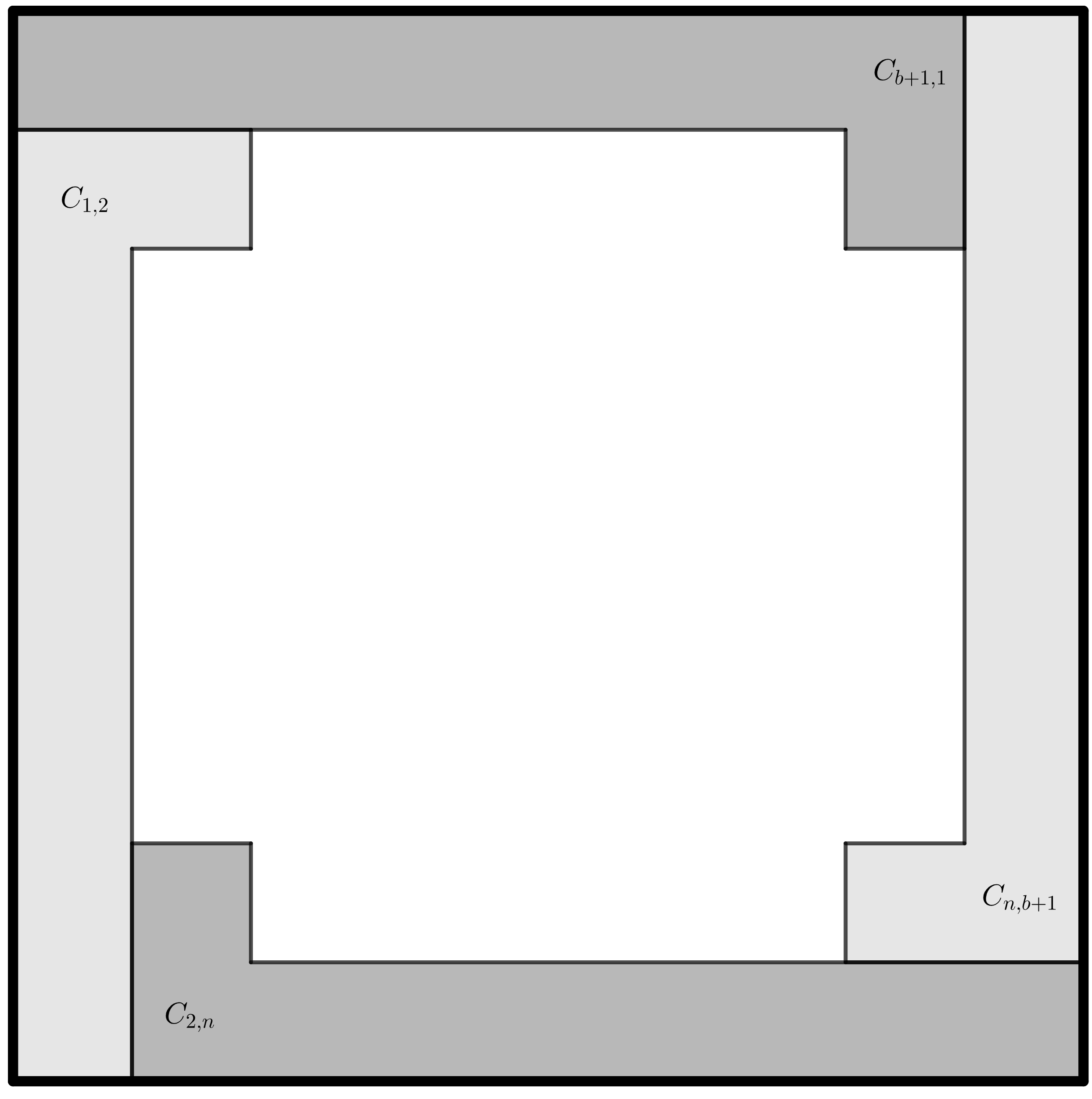}
                    \caption{Shows $\CPFR{L_{1, b}} \leq 4$. }
                    \label{fig:4LPacking}
                \end{figure}
            \end{proof}

    \subsection{$T$ Polyominoes}
    
        A $T$ polyomino is a natural extension of the $L$ polyomino.  
            \begin{definition}\label{def:T polyomino}
                A \textit{$T$ polyomino}, $T_{a,b}$, is a polyomino such that $T_{a,b} = \{ C_{i,1}: 1 \leq i \leq 2a+1\} \cup \{(a+1,j):2 \leq j \leq b+1 \}$. Notice $|T_{a,b}| = 2a + b + 1.$ We will define the \textit{anchor} of the polyomino to be the cell $C_{a+1,1}$.
            \end{definition}
        
        We will define $TR_{a,b}$, $TR^2_{a,b}$, and $TR^3_{a,b}$ to be a $90^{\circ}, 180^{\circ},$ and $270^{\circ}$ clockwise rotation of $T_{a,b}$ respectively. So we have:
        \begin{itemize}
            \item $TR_{a,b} = \{ C_{b+1,i} : 1 \leq i \leq 2a+1 \text{ and } (j,a+1) : 1 \leq j \leq b \}$ which will have anchor at cell $C_{b+1, a+1}$,
            \item $TR^2_{a,b} = \{ C_{i,b+1} : 1 \leq i \leq 2a+1 \text{ and } (a+1,j) : 1 \leq j \leq b \}$ which will have anchor at cell $C_{a+1, b+1}$, and 
            \item $TR^3_{a,b} = \{ C_{1,i} : 1 \leq i \leq 2a+1 \text{ and } (j,a+1) : 2 \leq j \leq b+1 \}$ which will have anchor at cell $C_{1, a+1}.$ 
        \end{itemize} 
        As before, we will use $T_{a,b}, TR_{a,b},$ etc. to describe the size and orientation of the polyomino and use the anchor to describe its location which is just a shift of $T_{a,b}, TR_{a,b}$, etc.. 
        

        To proceed, we will consider two cases for fixed $T$ polyominoes.  The first case, $b \leq 2a$ implies that the $T$ polyomino is as wide or wider than it is tall. The second case, $b \geq 2a+1$ implies that the $T$ polyomino is taller than it is wide.  

            
                

            \begin{theorem} \label{thm:fixed T polyomino b leq 2a} 
                For a \textit{$T$ polyomino}, $T_{a,b}$, if $b \leq 2a $ then $$\CPFI{T_{a,b}} = \left\lceil \frac{2a+1}{2b+1} \right\rceil.$$
            \end{theorem}
        
            \begin{proof}
                Place $T$ polyominoes, $T_{a,b}$, with anchors $C_{\left\lfloor \frac{b}{2}\right\rfloor+a+1, y}$, where $y = b+1 \mod{2b+1}$. In addition, if $n \mod {2b+1} \geq b+1$, add another $T$ polyomino with center $C_{\left\lfloor \frac{b}{2}\right\rfloor+a+1, n-(b+1)}$. 
                Note that this will put $T$ polyominoes, centered horizontally, so that the centers have exactly $2b$ rows between them, with the exception being at the bottom of the board where the centers may have less than $2b$ rows between them.  If there are at least $2b+1$ rows below the bottom most anchor, then we include another $T$ polyomino with anchor $C_{\left\lfloor \frac{b}{2}\right\rfloor+a+1, n-(b+1)}$. See Figure 
                \ref{fig:Tfixedbleq2a}.  
                
                \begin{figure}[h!]
                    \centering
                    \includegraphics[width = .3\linewidth]{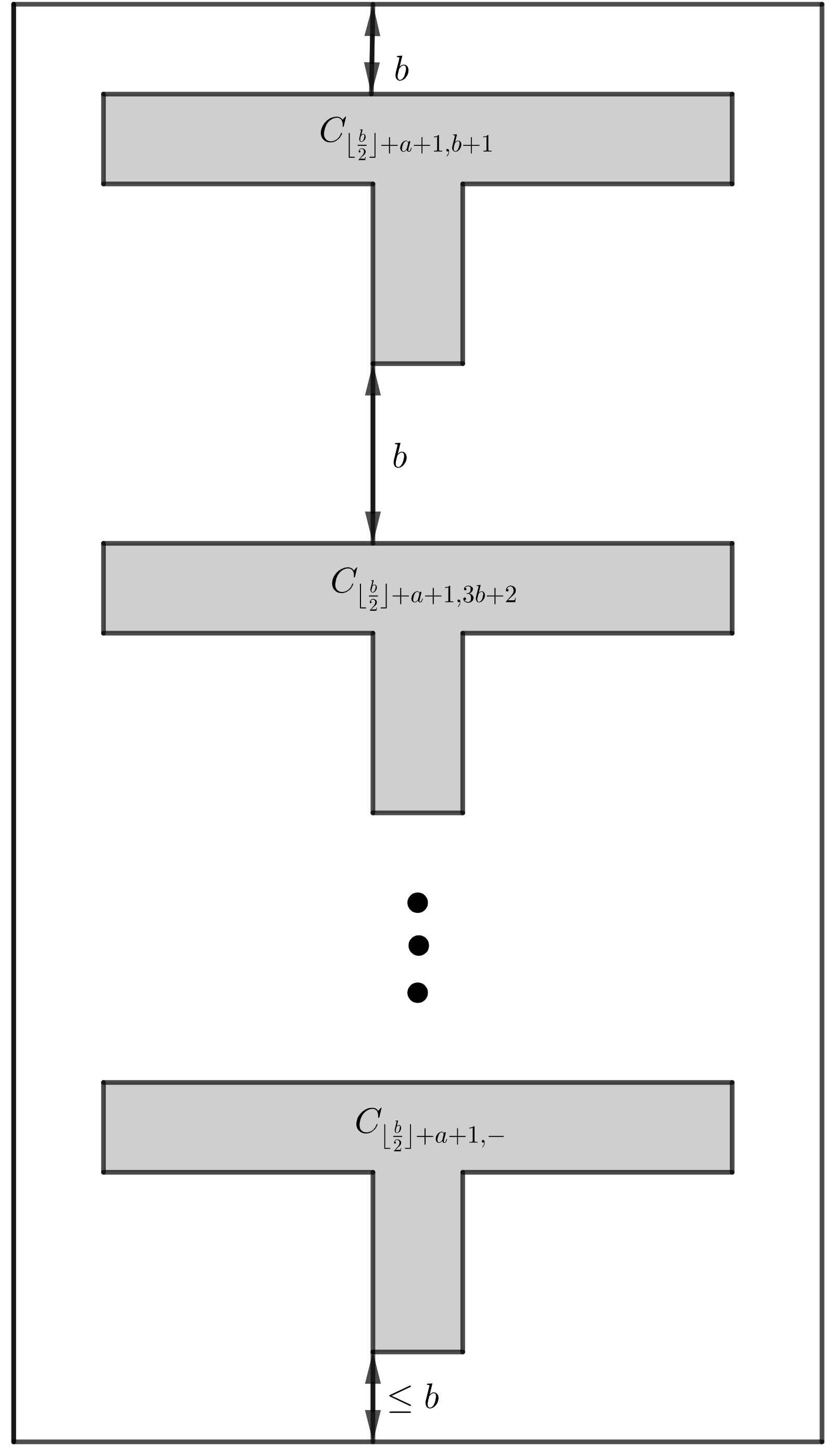}
                    \caption{Pattern for $\CPFI{T_{a,b}}$ when $b \leq 2a$. }
                    \label{fig:Tfixedbleq2a}
                \end{figure}
                
                Notice that there is no anchor in the top $b$ rows, and then every $2b+1$ rows thereafter will contain an anchor. If there are at least $2b+1$ empty rows at the bottom of the board, we include another $T$ polyomino. Therefore, the number of polyominoes that we now have in our arrangement is $\left\lceil \frac{n-b}{2b+1}\right\rceil = \left\lceil \frac{2a+1}{2b+1}\right\rceil$. This shows $$\CPFI{T_{a,b}} \leq \left\lceil \frac{2a+1}{2b+1} \right\rceil.$$
                
                To show this is optimal, notice that we must have an anchor in the top $b+1$ rows, Otherwise we can include another $T$ polyomino with an anchor in the top row. Next, notice that for every $2b+1$ rows thereafter, we must have an anchor.  If not, we will have at least $b+1$ consecutive empty rows. Therefore, the number of $T$ polyominoes must be at least 
                \begin{align*}
                    \CPFI{T_{a,b}} & \geq \left\lfloor \frac{n-(b+1)}{2b+1}\right\rfloor + 1\\
                        & = \left\lceil \frac{n-b}{2b+1}\right\rceil.\\
                \end{align*}
                The last equality follows since if $x$ and $y$ are positive integers then $\left\lfloor \frac{x}{y}\right\rfloor = \left\lceil\frac{x-y+1}{y} \right\rceil$.  
            \end{proof}
        
            \begin{theorem}\label{thm:fixed T polyomino b > 2a} 
                For a \textit{fixed $T$ polyomino}, $T_{a,b}$, if $b > 2a,$ then $$\CPFI{T_{a,b}}=\left\lceil \frac{b+1}{2a+1} \right\rceil.$$
            \end{theorem}
        
            \begin{proof}
                The proof here is similar to the proof of Theorem \ref{thm:fixed T polyomino b leq 2a} above, but consider anchors in columns rather than rows. 
            \end{proof}
        
        For free $T$ polyominoes, we can find an exact value when $a=b$ and bounds for $T_{a,b}$. 
            \begin{theorem}\label{thm:free T polyomino a=b}
                For a \textit{free $T$ polyomino}, $T_{a,b}$, if $a=b$, then $$\CPFR{T_{a,a}} = 2.$$
            \end{theorem}
        
            \begin{proof}
                    Assume $a=b$.  Consider placing a $T$ polyomino with anchor $C_{a+1, a+1}$ and another $TR$ polyomino with anchor $C_{2a+2, 2a+1}$. See Figure \ref{fig:TFreea=b}

                    Notice that this creates three pairwise disjoint regions which can not contain another $T$ polyomino.  Therefore $\CPFR{T_{a,a}} \leq 2$. 
                
                    Assume by way of contradiction that $\CPFR{T_{a,a}} =1$. Assume, without loss of generality, a single $T$ polyomino, unrotated, is placed on the board so that column $n$ is empty.  Notice that the anchor must be in the first $a+1$ rows and the first $2a$ columns. We can always place another $TR$ polyomino with anchor $C_{n, 2a+1}$ since column $n$ was empty and there are at least $a$ columns between the anchors. Note that column 1 or column $n$ must be empty and if column 1 is empty, the argument is similar. Therefore $\CPFR{T_{a,a}} \geq 2$ which completes the proof. 
                    \begin{figure}[h!]
                        \centering
                        \includegraphics[width = .4\linewidth]{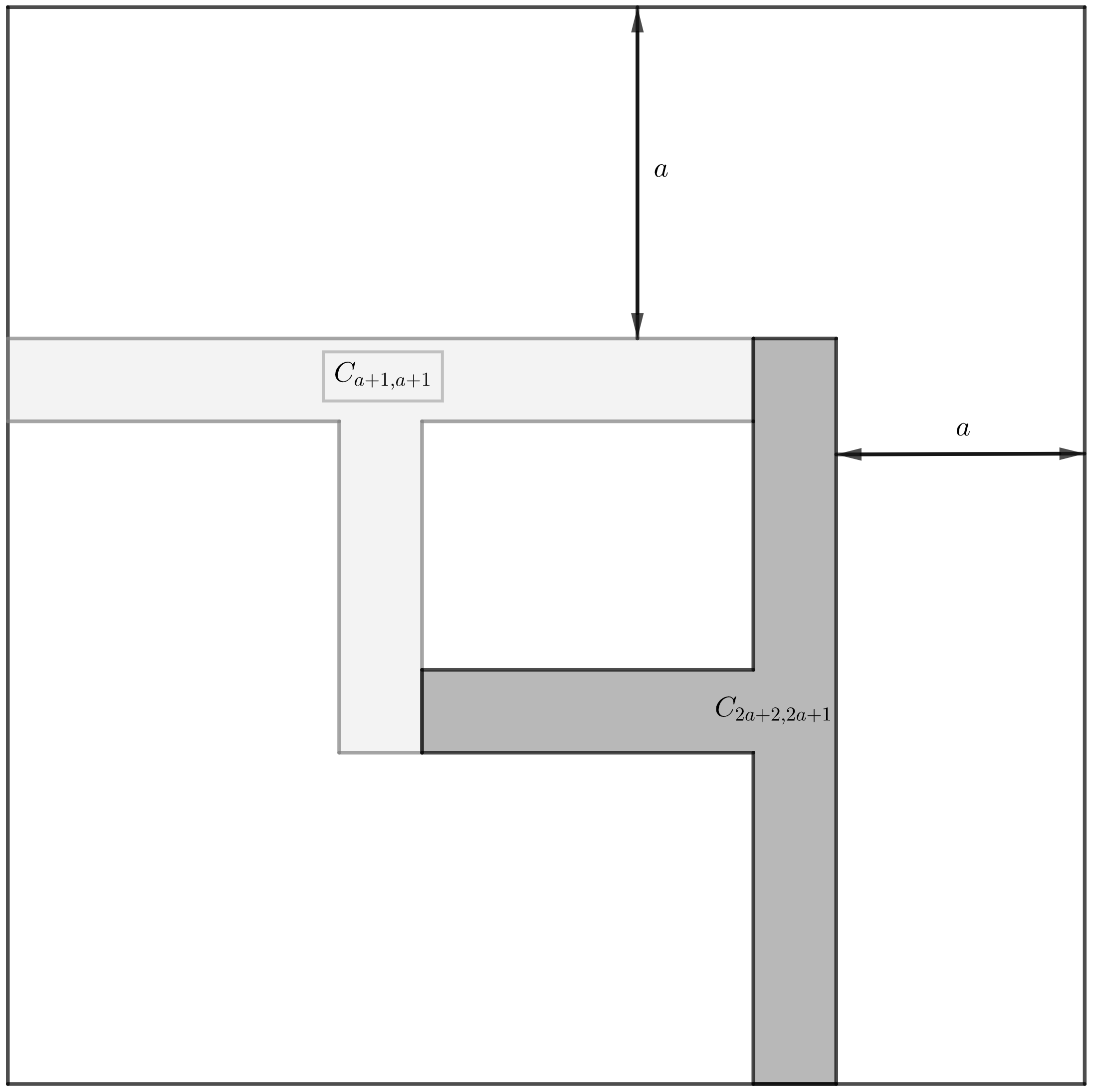}
                        \caption{Shows $\CPFR{T_{a,a}}=2$.  }
                        \label{fig:TFreea=b}
                    \end{figure}
            \end{proof}
                
            \begin{theorem}\label{thm:free T polyomino bounds}
                For a \textit{free $T$ polyomino}, $T_{a,b}$, $2 \leq \CPFR{T_{a,b}} \leq 4$.
            \end{theorem}
        
            \begin{proof}
                Assume by way of contradiction that $\CPFR{T_{a,b}} =1.$ Since $X_n$ or $X_1$ must be empty, we will assume without loss of generality a single $T$ polyomino, unrotated, is placed on the board with anchor $C_{x,y}$, so that $X_n$ is empty. 
                
                If $Y_n$ is empty, then we can place a $TR^2$ polyomino on the board with anchor $C_{x+1, n}$.  Therefore, the $T$ polyomino must intersect $Y_n$ so therefore the anchor must be on row $y=2a+1$. Since $X_n$ is empty, we can place a $TR$ polyomino with anchor $C_{n, a+1}$. This implies $\CPFR{T_{a,b}} \geq 2.$
                
                Let $c=\min\{a,b\}$. Consider placing the following four $T$ polyominoes on the board: $T_{a,b}$ with anchor $C_{a+1, c}$, $TR_{a,b}$ with anchor $C_{n-c+1, a+1}$, $TR^2_{a,b}$ with anchor $C_{n-a, n-c+1}$, and $TR^3_{a,b}$ with anchor $C_{c, n-a}$.
                   
                Case 1:  Assume $a <b$.
                Note that $T$ and $TR$ will always be adjacent.  This is because $T$ would occupy cells $\{C_{x,y}: 1 \leq x \leq 2a+1, y = a\}$ and $TR$ would occupy cells $\{C_{x,y}: a+2 \leq x \leq n-a, y=a+1\}$.  This implies that the four polyominoes will create five disjoint empty regions, four along the outside of the board and one on the interior of the board.  The empty interior region will be a square with side length $b-1$ and therefore we can not fit another $T$ polyomino in the region. Each empty region on the outside of the board will be subsets of a $a\times n$ rectangle and therefore can not fit another $T$ polyomino.  Thus, we have a clumsy packing with four $T$ polyominoes. See Figure \ref{fig:TFreea<b}.
                
                \begin{figure}[h!]
                    \centering
                    \includegraphics[width = .4\linewidth]{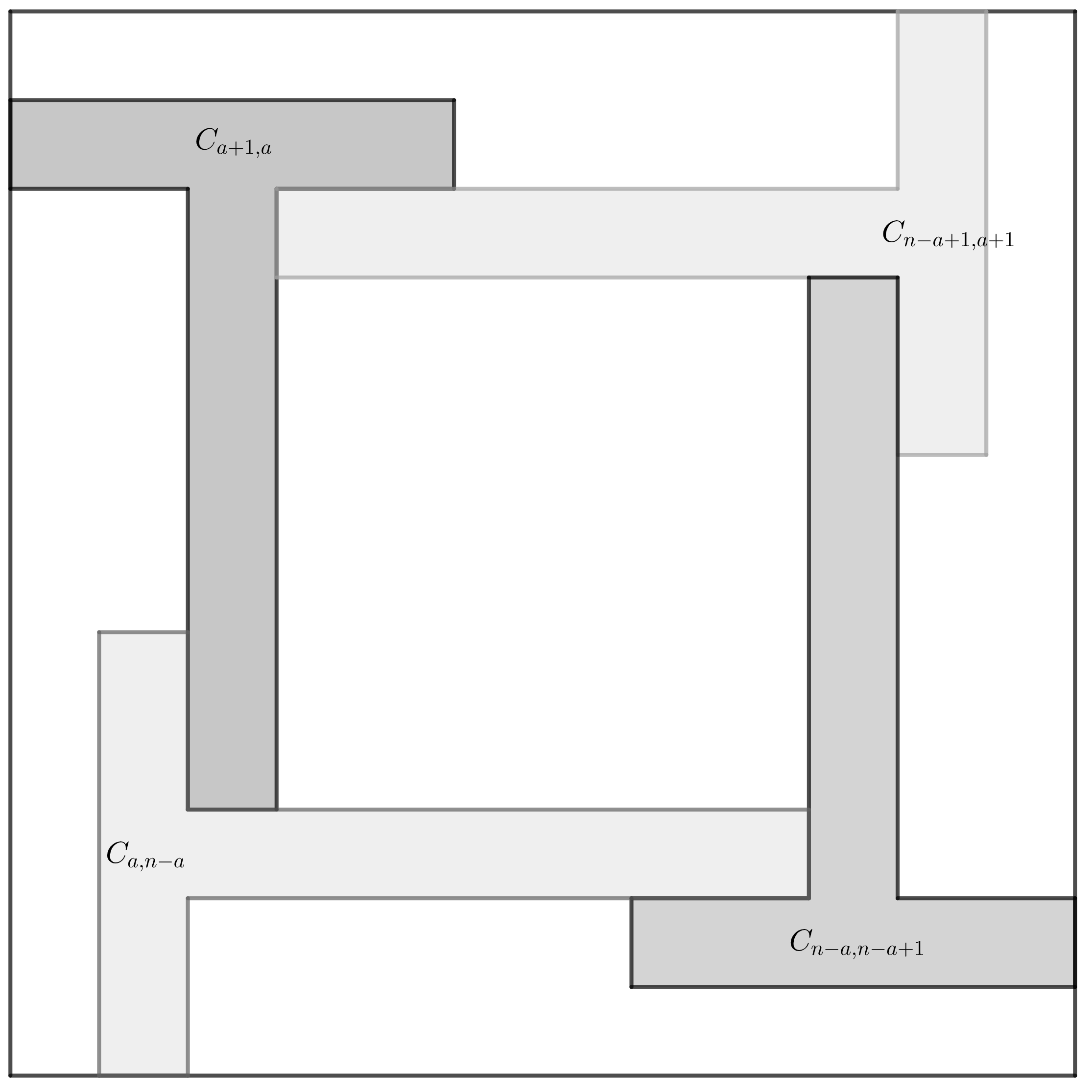}
                    \caption{Pattern for $\CPFR{T_{a,b}}$ when $a<b$.}
                    \label{fig:TFreea<b}
                \end{figure}
                
                Case 2:  Assume $b \leq a$. 
                Note that $T$ and $TR$ will always be adjacent.  This is because $T$ would occupy cells $\{C_{x,y}: 1 \leq x \leq 2a+1, y=b\}$ and $TR$ would occupy cells $\{C_{x,y}: x=n-c+1=2a+2, 1 \leq y \leq 2a+1\}$This implies that the four poloyominoes will create five disjoint empty regions, four along the outside of the board and one on the interior of the board.  The empty interior region will be a subset of a square with side length $n-2b = 2a+1-b \leq 2a$ and therefore we can not fit another $T$ polyomino in the region.  Each empty region on the outside of the board will be subsets of a $(b-1) \times 2a+1$ rectangle and therefore can not fit another $T$ polyomino.  Thus, we have a clumsy packing with four $T$ polyominoes. See Figure \ref{fig:TFreebleqa}.
                
                \begin{figure}[h!]
                    \centering
                    \includegraphics[width = .4\linewidth]{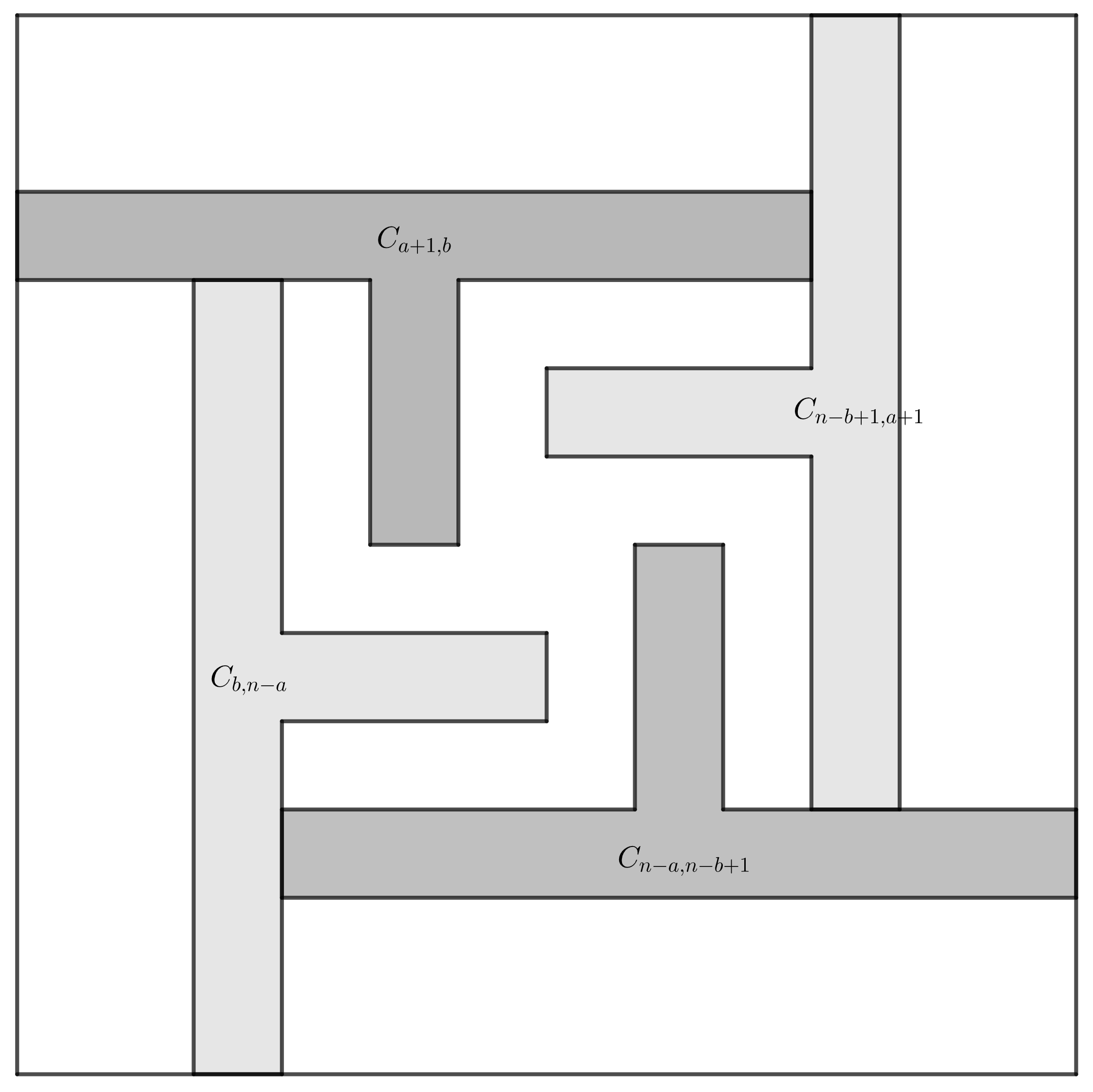}
                    \caption{Pattern for $\CPFR{T_{a,b}}$ when $b \leq a$. }
                    \label{fig:TFreebleqa}
                \end{figure}
            \end{proof}
            
            Note that these are not always optimal packing. For example, Theorems \ref{thm:free T polyomino a=b} shows that $\CPFR{T_{a,a}} = 2$.  Also, consider $a =4$ and $b = 3$. This can be packed using only three $T_{4,3}$ polyominoes, specifically two $T_{4,3}$ polyomino with centers $C_{5, 4}$ and $C_{5, 9}$ as well as $TR_{4,3}$ with center $C_{10, 8}$.  See the Figure \ref{fig:T43Optimal} for more details. 
            
            \begin{figure}[h!]
                \centering
                \includegraphics[width = .4\linewidth]{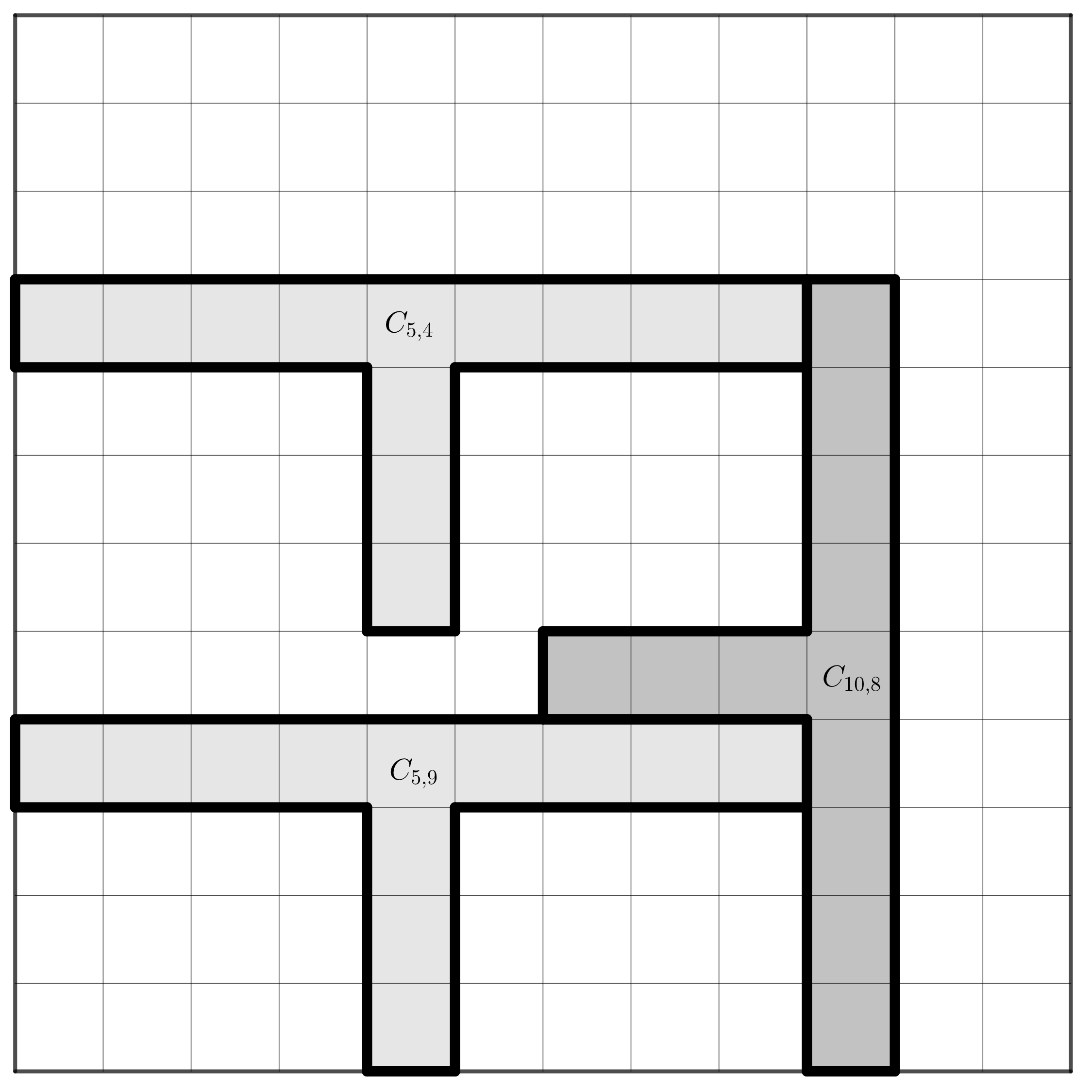}
                \caption{Shows $\CPFR{T_{4,3}} \leq 3$. }
                \label{fig:T43Optimal}
            \end{figure}
    
    \subsection{Plus Polyominoes}
        
        Our last polyomino is a \textit{plus} polyomino.  
    
            \begin{definition}\label{def:plus polyomino}
                A \textit{$plus$ polyomino}, $P_a$, is a polyomino such that $P_a = \{C_{i,a+1}: 1 \leq i \leq 2a+1\} \cup \{C_{a+1,j}: 1 \leq j \leq 2a+1 \}$. The $|P_a| = 4a + 1$. We will define the anchor of the polyomino to be the cell $C_{a+1, a+1}$. 
            \end{definition}
            
            A \textit{$plus$ polyomino} when rotated by an integer multiple of $90\deg$ remains the same polyomino. Thus a \textit{free,} and \textit{fixed,} polyominoes are indistinguishable. Therefore we will use $\CP{P_a}$ to denote the clumsy packing number since $\CPFI{P_a} = \CPFR{P_a}$

        
            \begin{theorem}\label{thm:cp of plus polyomino is 1} 
                For a $plus$ polyomino, $P_{a}$, $cp(P_a) = 1.$
            \end{theorem}
            
            \begin{proof}
                Let $P_a$ denote a $plus$ polyomino. Note that no polyomino can have an anchor in the leftmost $a$ columns, rightmost $a$ columns, top $a$ rows, or the bottom $a$ rows. Place the polyomino $P_a$ on the board with anchor at cell $C_{2a+1, 2a+1}$ and call this polyomino $P$. To place another $plus$ polyomino anchor at the board we must have at least $2a+1$ consecutive empty cells in a single column. This only happens in the leftmost $a$ columns and the rightmost $a$ columns which can not contain an anchor.  Therefore $\CP{P_a} = 1$. 
               
            \end{proof}
        
            
        
\section{Open Conjectures and Questions}\label{sect:questions}
    The area of packing polyominoes still has a lot of work. Although infinite spaces have been considered, there are a lot of applications for finite packing which makes this a worthwhile problem to pursue.  In the future, in addition to looking into different types of polyominoes and different sized finite boards, we propose some of the following questions and conjectures. Some of these may be much easier than others. 
    
    \subsection{Rectangular Polyominoes}
    
        \begin{question}\label{ques:free rectangular polyominoes}
            For a free rectangular polyomino, $R_{a,b}$, when $a \ne b$ and $a$ or $b \ne 1$, what is the $\CPFR{R_{a,b}}$? This is particularly interesting considering that a horizontal tiling of $R_{3, 6}$ shows $\CPFR{R_{3,6}} \leq \CPFI{R_{3, 6}} \leq 8$ and we think this is an equality. However, a similar horizontal tiling would imply $\CPFR{R_{3,9}} \leq \CPFI{R_{3, 9}} \leq 10$ but we have constructed a different tiling so that $\CPFR{R_{3,9}} \leq 9$ which keeps $7\times8$ rectangular regions empty in the four corners.  
        \end{question}
    
        
    
    \subsection{$L$ Polyominoes}
    
        \begin{conjecture}\label{conj:fixed L polyominoes} 
                For a fixed $L$ polyomino, $L_{a,b}$, if $b<a$, then $\CPFI{L_{a,b}} \leq \left\lfloor \frac{n}{(a+2)^2 + a} \right\rfloor (a+1) + \left \lceil \frac{n- \left \lceil \frac{n}{(a+1)^2 + a} \right \rceil ((a+1)^2) - a}{a+1} \right\rceil$.
        \end{conjecture}
        
        \begin{question}\label{ques:better bounds L polyomino}
            For a free $L$ polyomino, $L_{a,b}$, under what conditions do $a$ and $b$ need to be for $\CPFR{L_{a,b}}=2$, $\CPFR{L_{a,b}}=3$, $\CPFR{L_{a,b}}=4$, and $\CPFR{L_{a,b}}=5$?
        \end{question}
    
    \subsection{T Polyominoes}
    
        \begin{question} \label{ques:better bounds T polyomino}
            For a free $T$ polyomino, $T_{a,b}$, under what conditions do $a$ and $b$ need to be for $\CPFR{T_{a,b}}=2$, $\CPFR{T_{a,b}}=3$, and $\CPFR{T_{a,b}}=4$?
        \end{question}
        
        \begin{question}\label{ques:irregular T polyominoes}
            If $T$ polyominoes, $T_{a,b}$ has an addition parameter, $c$, such that $T_{a,b,c} = \{ (i,1): 1 \leq i \leq a+b+1\} \cup \{(a+1,j): 2 \leq j \leq c+1 \}$, how would the $\CPFR{T_{a,b,c}}$ and $\CPFI{T_{a,b,c}}$ be affected?
        \end{question}
    
    \subsection{Plus Polyominoes}
    
    \begin{question}\label{ques:cross polyominoes}
        If $plus$ polyominoes, $P_a$, have three addition variables, $b,c$ and $d$, such that $P_{a,b,c,d} = \{(i,a+1): 1 \leq i \leq 2a+1\} \cup \{(a+1,j): 1 \leq j \leq 2a+1 \} \cup \{(i,0): -a \leq i \leq b; (0,j): (0,j): -d \leq j \leq c \}$, how would the $cp(P_{a,b,c,d})$ be affected?
    \end{question}

    \bibliographystyle{abbrv}
    \bibliography{References}

\end{document}